\documentclass[11pt,reqno]{amsart}

\usepackage{amsbsy,amsfonts,amsmath,amssymb,amscd,amsthm,mathrsfs,verbatim,calc}
\usepackage{graphicx,epsfig,color}
\usepackage{enumitem}
\usepackage[a4paper,text={6.1in,8in},centering,includefoot,foot=1in]{geometry}

\small\normalsize
\theoremstyle{plain}
\newtheorem{mainthm}{Theorem}

\newtheorem{theorem}{Theorem}[section]
\newtheorem{lemma}[theorem]{Lemma}

\newtheorem{proposition}[theorem]{Proposition}
\newtheorem{corollary}[theorem]{Corollary}
\newtheorem{definition}[theorem]{Definition}

\theoremstyle{definition}

\newtheorem{remark}[theorem]{Remark}

\newtheorem{example}[theorem]{Example}
\newtheorem{question}[theorem]{Question}
\newtheorem{notation}[theorem]{Notation}

\numberwithin{equation}{section}
\let\oldmarginpar\marginpar
\renewcommand\marginpar[1]{\-\oldmarginpar[\raggedleft\footnotesize \textcolor{red}{#1}]{\raggedright\footnotesize\textcolor{red}{#1}}}

\begin{document}
\title[Powers of the edge ideals and matchings in hypergraphs]{Powers of the edge ideals and matchings in hypergraphs}

\author[F. Khosh-Ahang]{Fahimeh Khosh-Ahang Ghasr}
\address{{Department of Mathematics, School of Science, Ilam University,
P.O.Box 69315-516, Ilam, Iran.}}
\email{f.khoshahang@ilam.ac.ir and fahime$_{-}$khosh@yahoo.com}

\begin{abstract}
In this work, some combinatorial lower bound for regularity of powers of the edge ideal of a uniform hypergarph is gained. A family of hypergraphs whose regularity of edge ideal attains this bound and has a significant difference from the lower bounds heretofore obtained have also been introduced.
\end{abstract}




\maketitle
\setcounter{tocdepth}{1}

\section{Introduction}
In what follows  $t,d\in \mathbb{N}$, $d\geq 2$, $R=K[x_1, \dots , x_n]$ is the polynomial ring over the field $K$ and $S_1, \dots , S_m$ are monomials in $R$. For our convenience, sometimes we use the same notation $S_k$ for both of the monomial $S_k$ and $\{x_i \ :\  x_i \mid S_k\}$. By identifying the vertex $x_i$ and the variable $x_i$,  $\mathcal{H}$ stands for a simple hypergraph with $V(\mathcal{H})=\{x_1, \dots , x_n\}$ and $\mathcal{E}(\mathcal{H})=\{S_1, \dots , S_m\}$. Recall a hypergraph is $d$-\textbf{uniform} if all of its edges have the same cardinality $d$ and
$$I=\langle S_k \ : \ 1\leq k \leq m \rangle ,$$
is the edge ideal of $\mathcal{H}$.
So there is a correspondence between $d$-uniform hypergraphs and square-free monomial ideals in degree $d$ (see \cite{Villarreal}).

For a minimal graded free resolution
$$\cdots \longrightarrow \oplus_jR(-j)^{\beta_{i,j}} \longrightarrow \cdots \longrightarrow \oplus_jR(-j)^{\beta_{1,j}} \longrightarrow R \longrightarrow R/I \longrightarrow 0,$$
 of $R/I$, $\beta_{i,j}(R/I)$ is called the $(i,j)$th \textbf{graded Betti number} of $R/I$. Also, recall that the \textbf{regularity} of  $R/I$ is defined as
$$\mathrm{reg}(R/I) = \max\{j-i \ : \ \beta_{i,j}(R/I)\neq 0\}.$$

In the last few decades, studying the regularity of square-free monomial ideals because of their connections to other fields such as combinatorics, algebraic topology and computational algebra, has been extensively expanded. In particular, many researchers in combinatorial commutative algebra are interested in computing or bounding the regularity of the edge ideal of graphs and hypergraphs and their powers (see \cite{BBH} and \cite{Ha} for some surveys in this context).

In this paper, we firstly investigate free resolutions of powers of a square-free monomial ideal supported on a simplicial complex. Then we use the gained results in Section 2 for a simplicial complex introduced in \cite{Faridi2} for powers of the edge ideal of a hypergraph. These yield to characterize or bound the graded Betti numbers and regularity of $R/I^t$ in Section 3.
For instance in Corollary \ref{Cor3.7} a combinatorial characterization is gained for vanishing of $\beta_{i,2di}(R/I^2)$ in terms of the number of matchings of size $i$ in a $d$-uniform hypergraph $\mathcal{H}$. The needed definitions of some kinds of matchings are presented in Definition \ref{maindef}. Moreover the main result of this note is the following result:

\begin{mainthm}\label{3.7}
\begin{enumerate}
\item Suppose that $\mathcal{H}$ is a hypergraph which has a self semi-induced matching, say $\mathcal{S}=\{S_1, \dots, S_i\}$, of type $(i,j)$. Then  for each $\ell$ with $1\leq \ell \leq i$ we have $\beta_{i,|S_\ell|(t-1)+j}(R/I^t)\neq 0$ and so
$$\mathrm{reg}(R/I^t) \geq |S_\ell|(t-1)+(|\bigcup_{1\leq k \leq i}S_k|-i).$$

\item   Let $s$ be the number of all self semi-induced matchings in $\mathcal{H}$ of type $(i,j)$. Then
$$\beta_{i,j}(R/I)\geq s,$$
and if moreover $\mathcal{H}$ is $d$-uniform, then for all integers $t>1$
$$\beta_{i,d(t-1)+j}(R/I^t)\geq si.$$

\item Suppose that $\mathcal{H}$ is a $d$-uniform hypergraph. Then for all $t\in \mathbb{N}$
$$d(t-1)+(d-1)(i.m)_{\mathcal{H}}\leq d(t-1)+(s.s.i.m)'_{\mathcal{H}}\leq \mathrm{reg}(R/I^t).$$

\end{enumerate}
\end{mainthm}
 The above result is indeed a generalization of some parts of \cite[Theorem 6.5]{Ha+Vantuyl},  \cite[Lemma 2.2 and Proposition 2.5]{Katzman},  \cite[Theorems 3.5(1), 3.7(1) and Corollary 3.6]{SF} for arbitrary powers of the edge ideal of a hypergraph. Also, it gives a lower bound better than \cite[Theorem 3.7(1) and Corollary 3.8(1)]{BCKMS} and \cite[Corollary 3.9]{M+V} for the regularity of powers of the edge ideals of uniform hypergraphs, since there are examples of $d$-uniform hypergraphs achieving our bound such that $(s.s.i.m)'_{\mathcal{H}}$ is reasonably larger that $(d-1)(i.m)_{\mathcal{H}}$ as presented in Example \ref{exam2}.

\section{Preliminaries}
In this section, for each $1\leq k \leq m$,
 $$S_k=\prod_{1\leq i \leq n} x_i^{a_{i,k}}, \textbf{a}_k=(a_{1,k}, \dots , a_{n,k})\in \{0,1\}^n, \sum_{1\leq i \leq n} a_{i,k}=d .$$
 Hence $I$ is a square-free monomial ideal of $R$ generated in degree $d$. Now assume that
$$\mathcal{B}_{m,t}=\{\textbf{b} \ : \ \textbf{b}=b_1\textbf{e}_1+\dots +b_m\textbf{e}_m=(b_1, \dots , b_m)\in \{0, \dots , t\}^m, \sum_{1\leq \ell\leq m}b_{\ell}=t \}.$$
Hereafter we order the elements of $\mathcal{B}_{m,t}$ in such a way that if $\textbf{b}=(b_1, \dots , b_m)$ and \linebreak $\textbf{b}'=(b'_1, \dots , b'_m)$ are two elements of $\mathcal{B}_{m,t}$, then $\textbf{b}<\textbf{b}'$ if and only if there exists an integer $1\leq k \leq m$ such that for all $1\leq \ell \leq k-1$ we have $b_\ell=b'_\ell$ and $b_k>b'_k$.
If we set $S^{\textbf{b}}=S_1^{b_1}\dots S_m^{b_m}$ for all $\textbf{b}=(b_1, \dots , b_m)\in \mathcal{B}_{m,t}$, then
$I^t=\langle S^\textbf{b} \ : \ \textbf{b}\in \mathcal{B}_{m,t} \rangle$.
But note that $\{ S^\textbf{b} \ : \ \textbf{b}\in \mathcal{B}_{m,t} \}$ is not necessarily a minimal set of generators of $I^t$.
So, $I^t$ is minimally generated by $p_{m,t}$ elements, when $p_{m,t}\leq |\mathcal{B}_{m,t}|= \left( \begin{array}{c}
t+m-1 \\
m-1
\end{array} \right)$.
In this section we are going to study some simplicial resolutions of $R/I^t$. To this aim we first need some notation.

\begin{notation}\label{Notation}
For given $I$ and $t$ with described notions, suppose that $A$ is the $n\times m$ matrix with columns $\textbf{a}_1, \dots, \textbf{a}_m$, and $B$ is the $m\times p_{m,t}$ matrix with columns $\textbf{b}_1, \dots, \textbf{b}_{p_{m,t}}$. Also for all $\ell_1, \dots, \ell_i\in \{1,\dots, p_{m,t}\}$,
$B[\ell_1, \dots , \ell_i]$ is the $m\times i$ submatrix of $B$ with columns $\textbf{b}_{\ell_1}, \dots, \textbf{b}_{\ell_i}$ and $B[\ell_1, \dots ,\widehat{\ell_k},\dots, \ell_i]$ is the $m\times (i-1)$ submatrix of $B$ with columns \linebreak $\textbf{b}_{\ell_1}, \dots, \textbf{b}_{\ell_{k-1}},\textbf{b}_{\ell_{k+1}}, \dots,  \textbf{b}_{\ell_i}$ (of course after ordering $\ell_1, \dots , \ell_i$). Moreover for each real matrix $C=(c_{i,j})_{p\times q}$, we assign the vector
$$\mathrm{Max}(C)=\left(
\max\{c_{11}, \dots, c_{1q}\}, \dots, \max\{c_{p1}, \dots, c_{pq}\} \right)  .$$
Also recall that for each real vector $\textbf{c}=(c_1, \dots , c_n)$, the sum of $\textbf{c}$ is  $\mathrm{Sum}(\textbf{c})=c_1+\dots +c_n$.
\end{notation}

Suppose that $\Delta$ is a simplicial complex whose vertices are labelled by the minimal monomial generators of $I$ and whose faces are labelled by the least common multiple of the vertices of that face.  Then we say that  $\Delta$ \textbf{supports a free resolution} $F_\bullet$ of $R/I$ or $F_\bullet$ \textbf{is supported on} $\Delta$ if the labelling of the vertices of $\Delta$ satisfies certain properties and the simplicial chain complex of $\Delta$ can be homogenized using the monomial labels on the faces.
A well-known simplicial complex which supports a free resolution of $R/I$ is Taylor complex of $I$ which is a simplex on $m$ vertices labelled by the minimal monomial generators of $I$ and is denoted  by $\mathrm{Taylor}(I)$.

In view of \cite{B+P+S}, if $\Delta$ supports a free resolution $F_\bullet$ of $I^t$, then $\Delta$ is a subcomplex of $\mathrm{Taylor}(I^t)$   and $F_\bullet$ can be structured as follows.

$$F_\bullet : 0 \rightarrow F_{\mathrm{dim}\Delta+1} \rightarrow \dots \rightarrow F_{i+1} \stackrel {\partial_{i+1}}{\longrightarrow} F_i \stackrel {\partial_i}{\longrightarrow} F_{i-1}\rightarrow  \dots \rightarrow F_0 \rightarrow R/I^t \rightarrow 0,$$
 where $F_0=R$ and $F_i$ is the free $R$-module whose free generators are $e_\tau$s for all faces $\tau$ of $\Delta$ of dimension $i-1$. If the vertices of $\tau$ is labelled by the monomials $S^{\textbf{b}_{\ell_1}},\dots , S^{\textbf{b}_{\ell_i}}$, then we sometimes denote $e_\tau$ by $e_{\ell_1, \dots , \ell_i}$, where $1\leq \ell_1< \dots < \ell_i\leq p_{m,t}$. Also for all $1\leq i \leq \mathrm{dim}\Delta +1$ and $e_{\ell_1, \dots , \ell_i}\in F_i$,
$$\partial_i(e_{\ell_1, \dots , \ell_i})=\sum_{1\leq k \leq i} (-1)^k\mu_k e_{\ell_1, \dots , \widehat{\ell_k}, \dots, \ell_i},$$
where
$$\mu_k= \frac{\mathrm{lcm}(S^{\textbf{b}_{\ell_1}}, \dots , S^{\textbf{b}_{\ell_i}})}{\mathrm{lcm}(S^{\textbf{b}_{\ell_1}}, \dots ,\widehat{S^{\textbf{b}_{\ell_k}}},\dots,  S^{\textbf{b}_{\ell_i}})}.$$
Note that for each $1\leq j \leq i$, if $\textbf{b}_{\ell_j}=(b_{1,\ell_j}, \dots , b_{m,\ell_j})$, then we have
\begin{align*}
S^{\textbf{b}_{\ell_j}}&=S_1^{b_{1,\ell_j}}\dots S_m^{b_{m,\ell_j}}=\left(\prod_{1\leq i \leq n}x_i^{a_{i,1}}\right)^{b_{1,\ell_j}}\dots \left(\prod_{1\leq i \leq n}x_i^{a_{i,m}}\right)^{b_{m,\ell_j}}\\
&=\prod_{1\leq i \leq n}x_i^{a_{i,1}b_{1,\ell_j}+\dots +a_{i,m}b_{m,\ell_j}}=\prod_{1\leq i \leq n}x_i^{c_{i,\ell_j}},
\end{align*}
where $AB=(c_{r,s})_{n\times p_{m,t}}$. Now assume that $\mathrm{Max}(AB[\ell_1, \dots , \ell_i])=(c_1, \dots , c_n)$ and
$$\mathrm{Max}(AB[\ell_1, \dots, \widehat{\ell_k}, \dots , \ell_i])=(d_{1,k}, \dots , d_{n,k}).$$
Then we have
$\mu_k=\prod_{1\leq i \leq n}x_i^{c_i-d_{i,k}}$.
By considering the degree of $e_{\ell_1, \dots , \ell_i}$ as
$$\mathrm{deg}(e_{\ell_1, \dots , \ell_i})=\mathrm{deg}(\mathrm{lcm}(S^{\textbf{b}_{\ell_1}}, \dots , S^{\textbf{b}_{\ell_i}}))
=\mathrm{deg}(x_1^{c_1}\dots x_n^{c_n})
=\mathrm{Sum}(\mathrm{Max}(AB[\ell_1, \dots , \ell_i])),$$
$F_\bullet$ is a graded free resolution of $R/I^t$ which is not necessarily minimal as mentioned before, but we may use it for computing the graded Betti numbers $\beta _{i,j}(R/I^t)$ as follows.
\begin{align}\label{eq0}
\beta _{i,j}(R/I^t)&=\mathrm{dim}_K(\mathrm{Tor}_{i}^R(R/I^t,K))_j\\
\nonumber &=\mathrm{dim}_K(H_{i}(F_\bullet \otimes _R  R/\langle x_1, \dots , x_n\rangle ))_j\\
\nonumber &=\mathrm{dim}_K(\mathrm{Ker}\overline{\partial}_{i}/\mathrm{Im}\overline{\partial}_{i+1})_j.
\end{align}
After tensoring $F_\bullet$ with $R/\langle x_1, \dots , x_n\rangle$, for each $e_\tau$ where $\tau=\{S^{\textbf{b}_{\ell_1}}, \dots , S^{\textbf{b}_{\ell_i}}\}\in \Delta$, we have
\begin{equation}\label{eq1}
\overline{\partial}_i (\overline{e_{\ell_1, \dots , \ell_i}})=\sum_{1\leq k \leq i, \mathrm{Max}(AB[\ell_1, \dots , \ell_i])=\mathrm{Max}(AB[\ell_1, \dots ,\widehat{\ell_k},\dots , \ell_i])} (-1)^k \overline{e_{\ell_1, \dots, \widehat{\ell_k}, \dots, \ell_i}},
\end{equation}
where for each $0\leq i \leq \mathrm{dim}\Delta +1$ and each member $u\in F_i$, $\overline{u}$ is the natural image of $u$ in $\overline{F_i}=F_i\otimes_R R/\langle x_1, \dots , x_n\rangle$ and $\overline{\partial}_i =\partial_i\otimes_R \mathrm{id}_{R/\langle x_1, \dots , x_n\rangle }$.

 Afterwards suppose that $\Delta$ is a simplicial complex supporting a free resolution of $R/I^t$.

\begin{remark}\label{Remarks2.2}
In view of Equation (\ref{eq1}), the following statements hold:

\begin{enumerate}
\item  $\overline{e_{\ell_1, \dots , \ell_i}}\in \mathrm{Ker}\overline{\partial}_i$ if and only if $\{S^{\textbf{b}_{\ell_1}}, \dots, S^{\textbf{b}_{\ell_i}}\}\in \Delta$ and for all $1\leq k \leq i$ we have
$$\mathrm{Max}(AB[\ell_1, \dots , \ell_i])\neq \mathrm{Max}(AB[\ell_1, \dots , \widehat{\ell_k}, \dots , \ell_i]).$$

\item If $\overline{e_{\ell_1, \dots , \ell_i}}\in \mathrm{Im}\overline{\partial}_{i+1}$, then there exists a face $\{S^{\textbf{b}_{\ell_1}}, \dots , S^{\textbf{b}_{\ell_{i+1}}}\}$ of $\Delta$ of dimension $i$ such that
$\mathrm{Max}(AB[\ell_1, \dots , \ell_{i+1}])=\mathrm{Max}(AB[\ell_1, \dots , \ell_i])$.

\item If there exists a face $\{S^{\textbf{b}_{\ell_1}}, \dots, S^{\textbf{b}_{\ell_{i+1}}}\}$ of $\Delta$ of dimension $i$ such that
$$\mathrm{Max}(AB[\ell_1, \dots , \ell_{i+1}])=\mathrm{Max}(AB[\ell_1, \dots , \ell_i]),$$
and for all $1\leq k \leq i$, we have
$\mathrm{Max}(AB[\ell_1, \dots ,\widehat{\ell_k}, \dots, \ell_{i+1}]) \neq\mathrm{Max}(AB[\ell_1, \dots , \ell_i])$,
then $\overline{e_{\ell_1, \dots , \ell_i}}\in \mathrm{Im}\overline{\partial}_{i+1}$.

\item It can be easily seen that $\mathcal{B}_{m,1}=\{\textbf{e}_1, \dots , \textbf{e}_m\}$ and so $B=I_m$ where $t=1$. Hence $AB=A=(a_{i,j})_{n\times m}$ is a matrix with entries $0$ or $1$, in which the $i$th row associates to the variable $x_i$ and the $j$th column associates to the generator $S_j$ of $I$. In fact $a_{i,j}=1$ if $x_i \mid S_j$ and else $a_{i,j}=0$.  So, in each column there exist exactly $d$ entries $1$. Also for each $i$, the $i$th row has at least one $1$, which lies in columns $j$   with $x_i\mid S_j$. Therefore\linebreak $\mathrm{Max}(AB)=(1, 1, \dots , 1)$ and
$\mathrm{Max}(AB[\ell_1, \dots , \ell_i])=(c_1, \dots, c_n)$, where $c_j=1$ if \linebreak
$x_j\in \bigcup_{1\leq r \leq i}S_{\ell_r}$,
 else $c_j=0$. This shows that for each $1\leq k \leq i$,
$$\mathrm{Max}(AB[\ell_1, \dots , \ell_i])\neq \mathrm{Max}(AB[\ell_1, \dots , \widehat{\ell_k}, \dots , \ell_i]),$$
 if and only if
 $S_{\ell_k}\nsubseteq \bigcup_{1\leq r \leq i, r\neq k} S_{\ell_r}$.

\end{enumerate}
\end{remark}
The above remarks illustrate that one can evaluate the graded Betti numbers of the edge ideal of a hypergraph by interaction of its edges as you see for instance in \cite[Theorem 6.5]{Ha+Vantuyl}, \cite[Lemma 2.2]{Katzman} and \cite[Section 3]{SF}. Note that all of these results just use Taylor resolution of $I^t$ (when $t=1$) which has easier structure as you may see above.

One can naturally generalize \cite[Lemma 3.2]{SF} and obtain the following result, which is needed in the next section.

\begin{lemma}\label{Lemma2.3} Let $i,j$ be integers. Set $$\mathcal{L}_{i,j}=\{e \ : \ e=\overline{e_{\ell_1, \dots , \ell_i}}\in \mathrm{Ker}\overline{\partial}_i\setminus \mathrm{Im}\overline{\partial}_{i+1}, \mathrm{Sum}(\mathrm{Max}(AB[\ell_1, \dots , \ell_i]))=j \}.$$

\begin{enumerate}
\item If for all faces $\{S^{\textbf{b}_{\ell_1}}, \dots , S^{\textbf{b}_{ \ell_i}}\}$ of $\Delta$ with $\mathrm{Sum}(\mathrm{Max}(AB[\ell_1, \dots , \ell_i]))=j$ we have $\mathrm{Max}(AB[\ell_1, \dots , \ell_i])\neq \mathrm{Max}(AB[\ell_1, \dots , \widehat{\ell_k}, \dots , \ell_i]),$ for all $1\leq k \leq i$, then  \linebreak$\beta_{i,j}(R/I^t)\leq |\mathcal{L}_{i,j}|$.

\item If for all faces $\{S^{\textbf{b}_{\ell_1}}, \dots , S^{\textbf{b}_{ \ell_{i+1}}}\}$ of $\Delta$  with $\mathrm{Sum}(\mathrm{Max}(AB[\ell_1, \dots , \ell_i]))=j$ and \linebreak $\mathrm{Max}(AB[\ell_1, \dots , \ell_{i+1}])=\mathrm{Max}(AB[\ell_1, \dots , \ell_i])$, we have  $$\mathrm{Max}(AB[\ell_1, \dots ,\widehat{\ell_k},\dots, \ell_{i+1}])\neq\mathrm{Max}(AB[\ell_1, \dots , \ell_i]),$$ for all $1\leq k \leq i$, then $\beta_{i,j}(R/I^t)\geq |\mathcal{L}_{i,j}|$.
\end{enumerate}
\end{lemma}

\section{Powers of the edge ideal of a hypergraph}
To start, we recall the following definitions.
\begin{definition} (See \cite{Berge} and \cite[Definitions 2.1 and Notation 2.3]{SF}.)\label{maindef}
Let $\mathcal{S}=\{S_1, \dots , S_i\}$ be a family of edges of $\mathcal{H}$. We define the type of $\mathcal{S}$ as $(i,j)$, where $i$ is the cardinality of $\mathcal{S}$ and $j=|\bigcup_{1\leq \ell \leq i} S_\ell |$.

\begin{enumerate}
\item $\mathcal{S}$ is called a \textbf{matching} in $\mathcal{H}$ if for each $\ell , \ell'$ with  $1\leq \ell < \ell' \leq i$,  $S_\ell\cap S_{\ell '}=\emptyset$.

\item $\mathcal{S}$ is called a \textbf{self matching} in $\mathcal{H}$ if for all $k$ with $1\leq k \leq i$, $S_k\nsubseteq \bigcup _{1\leq \ell \leq i,\ell\neq k} S_\ell$.

\item $\mathcal{S}$ is called a \textbf{semi-induced matching} in $\mathcal{H}$ if for each $S\in \mathcal{E}(\mathcal{H})\setminus \{S_1, \dots , S_i\}$, $S\nsubseteq \bigcup_{1\leq \ell \leq i} S_\ell$.

\item $\mathcal{S}$ is called a \textbf{self semi-induced matching} in $\mathcal{H}$ if $\mathcal{S}$ is both self matching and semi-induced matching.

\item $\mathcal{S}$ is called an \textbf{induced matching} in $\mathcal{H}$ if $\mathcal{S}$ is both matching and semi-induced matching.

\item  We use the following notions:
\begin{align*}
&m_{\mathcal{H}}=\max \{i \ : \  \textrm{there is a matching of size  } i \textrm{ in } \mathcal{H} \};\\
&(i.m)_{\mathcal{H}}=\max \{i \ : \  \textrm{there is an induced matching of size  } i \textrm{ in } \mathcal{H} \};\\
&(i.m)'_{\mathcal{H}}=\max \{j-i \ : \  \textrm{there is an induced matching of type  }(i,j) \textrm{ in } \mathcal{H}\};\\
&(s.s.i.m)_{\mathcal{H}}=\max \{i \ : \  \textrm{there is a self semi-induced matching of size  } i \textrm{ in } \mathcal{H}  \};\\
&(s.s.i.m)'_{\mathcal{H}}=\max \{j-i \ : \  \textrm{there is a self semi-induced matching of type  }(i,j) \textrm{ in } \mathcal{H}  \};\\
&(s.i.m)'_{\mathcal{H}}=\max \{j-i \ : \  \textrm{there is a semi-induced matching of type  }(i,j) \textrm{ in } \mathcal{H}  \}.\\
\end{align*}

\end{enumerate}
\end{definition}
The following observations may be helpful in the sequel.
\begin{remark}\label{Remark3.2}
Suppose that $d$ is the maximum size of edges in $\mathcal{H}$.
\begin{enumerate}
\item The following inequalities are straightforward:
$$(i.m)_{\mathcal{H}}\leq \min\{m_\mathcal{H},(s.s.i.m)_{\mathcal{H}}\} , \ \  (i.m)'_{\mathcal{H}}\leq (s.s.i.m)'_{\mathcal{H}}\leq (s.i.m)'_{\mathcal{H}}.$$
Also, if $\mathcal{S}=\{S_1, \dots , S_i\}$ is an induced matching of type $(i,j)$ in $\mathcal{H}$, then
$$j-i=|\bigcup_{1\leq \ell \leq i} S_\ell |-i=\sum_{1\leq \ell \leq i}|S_\ell |-i\leq (d-1)i\leq (d-1)(i.m)_{\mathcal{H}}.$$
Therefore $(i.m)'_{\mathcal{H}}\leq (d-1)(i.m)_{\mathcal{H}}$. Furthermore the equality holds when there exists an induced matching of maximum size containing $d$-sets. Hence in a $d$-uniform hypergraph
$$(d-1)(i.m)_{\mathcal{H}}=(i.m)'_{\mathcal{H}}\leq (s.s.i.m)'_{\mathcal{H}}\leq (s.i.m)'_{\mathcal{H}}.$$
Thus in view of \cite[Proposition 2.7]{Moradi+Khosh2017} for any $d$-uniform simple hypergraph $\mathcal{H}$ such that for each pair of distinct edges $E$ and $E'$, $E\cap E'\neq \emptyset$ implies $|E\cap E'|=d-1$ (in particular for simple graphs) we have
$$(d-1)(i.m)_{\mathcal{H}}=(i.m)'_{\mathcal{H}}= (s.s.i.m)'_{\mathcal{H}}= (s.i.m)'_{\mathcal{H}}.$$

\item The matrix $A=(a_{i,j})_{n\times m}$, defined in Notation \ref{Notation}, is the well-known incidence matrix of $\mathcal{H}$.

\item  $B=(b_{i,j})_{m\times p_{m,t}}$, defined in Notation \ref{Notation}, is a matrix in which the rows are labelled by $S_i$s and the columns are labelled by elements of $\mathcal{B}_{m,t}$. So, $AB=(c_{i,j})_{n\times p_{m,t}}$, where $c_{i,j}=\sum_{1\leq k \leq m, x_i\in S_k} b_{k,j}$.

\end{enumerate}
\end{remark}

\begin{lemma}\label{lemma3.3}
Suppose that $k\in \mathbb{N}$ and $\mathcal{S}=\{S_1, \dots , S_i\}$ is a self semi-induced matching in $\mathcal{H}$. Then for each (not necessarily distinct integers) $1\leq u_1, \dots , u_k\leq i$, $S_{u_1}\dots S_{u_k}$ is a minimal monomial generator of $I^k$.
\end{lemma}
\begin{proof}
Assume on the contrary that $S_{u_1}\dots S_{u_k}$ is not a minimal generator of $I^k$ for some \linebreak $1\leq u_1, \dots , u_k\leq i$. Then there exist (not necessarily distinct) elements $r_1, \dots , r_k\in \{1, \dots , m\}$ such that $S_{r_1}\dots S_{r_k} \mid S_{u_1}\dots S_{u_k}$ and $\{r_1, \dots , r_k\}\neq \{u_1, \dots , u_k\}$. By omitting equal terms from both sides of the division, we have $S_r \mid S_{u_{j_1}}\dots S_{u_{j_s}}$ for some\linebreak $r\in \{r_1, \dots , r_k\}\setminus \{u_{j_1}, \dots , u_{j_s}\}$ where $\{u_{j_1}, \dots , u_{j_s}\}\subseteq \{u_1, \dots , u_k\}$. This means\linebreak $S_r\subseteq S_{u_{j_1}}\cup \dots \cup S_{u_{j_s}}$. This contradicts to $\mathcal{S}$ is a self semi-induced matching.
\end{proof}

Now we need  some definitions from \cite{Faridi1} and \cite{Faridi2}. In \cite{Faridi2} a simplicial complex $\mathbb{L}_m^t$ is defined and it is shown that it supports a free resolution of $I^t$. In \cite{Faridi1}, the case $t=2$ is investigated individually. So, for our next results of this section, we refer the reader to Definitions 3.1 and 3.4 in \cite{Faridi1} and  Definitions 4.2 and 5.1 and Proposition 4.3 in \cite{Faridi2}. Hereinafter we apply the notions in Section 2 for the simplicial complex $\mathbb{L}_m^t$.

\begin{proposition}\label{3.6}
Suppose that $\mathcal{H}$ is a $d$-uniform hypergraph,  $i>1, j=2di$ and
$$\mathcal{L}_{i,j}=\{e \ : \ e=\overline{e_{\ell_1, \dots , \ell_i}}\in \mathrm{Ker}\overline{\partial}_i\setminus \mathrm{Im}\overline{\partial}_{i+1}, \mathrm{Sum}(\mathrm{Max}(AB[\ell_1, \dots , \ell_i]))=j \}.$$
Then $\beta_{i,j}(R/I^2)=|\mathcal{L}_{i,j}|$.
\end{proposition}
\begin{proof}
In view of Lemma \ref{Lemma2.3}, it is enough to show that the assumptions of Parts 1 and 2 in Lemma \ref{Lemma2.3} hold, when $t=2$, $i>1$ and $j=2di$. To this end, suppose that \linebreak $\tau=\{S^{\textbf{b}_{\ell_1}}, \dots , S^{\textbf{b}_{ \ell_i}}\}$ is a face of $\mathbb{L}^2(I)$. By means of \cite[Definition 3.1]{Faridi1}, one may assume that $\tau$ has one of the following forms:
$$\tau=\{S_1^2, S_1S_2, \dots , S_1S_i\} \text{ or } \tau=\{S_{\ell_1}S_{\ell'_1}, \dots , S_{\ell_i}S_{\ell'_i}\},$$
when $\ell_k< \ell'_k$ for all $1\leq k \leq i$. Now in each case we investigate the degree of $e_{\ell_1, \dots , \ell_i}$. Suppose that  $\mathrm{Max}(AB[\ell_1, \dots , \ell_i])=(c_k)$. In view of Remark \ref{Remark3.2}(3), in the first case, $c_k=2$ if $x_k\in S_1$, $c_k=1$ if $x_k\in \bigcup_{2\leq \ell \leq i} S_\ell\setminus S_1$ and else $c_k=0$. So
\begin{align*}
\mathrm{deg}( e_{\ell_1, \dots ,  \ell_i})&=\sum c_k\\
&=2d+|\bigcup_{2\leq \ell \leq i} S_\ell\setminus S_1|\\
&\leq 2d +(i-1)d\\
&=(i+1)d<2di.
\end{align*}
Hence $\tau$ should be in the form of the second one. Now, in the second case $c_k=2$  if \linebreak $x_k\in \bigcup_{1\leq r \leq i}(S_{\ell_r}\cap S_{\ell'_r})$, $c_k=1$ if $x_k\in \bigcup_{1\leq r \leq i}((S_{\ell_r}\cup S_{\ell'_r})\setminus (S_{\ell_r}\cap S_{\ell'_r}))$ and else $c_k=0$. Hence
$\mathrm{deg}( e_{\ell_1, \dots ,\ell_i})=\sum c_k\leq 2di,$
and the equality holds when $(S_{\ell_k}\cup S_{\ell'_k})$s are disjoint for all $1\leq k \leq i$. So,
for all $1\leq k \leq i$,  the rows associated to the vertices in $S_{\ell_k}\cup S_{\ell'_k}$ will be zero in $\mathrm{Max}(AB[\ell_1, \dots , \widehat{\ell_k}, \dots, \ell_i])$, while they have entries 1 or 2 in $\mathrm{Max}(AB[\ell_1, \dots, \ell_i])$. Thus $\mathrm{Max}(AB[\ell_1, \dots , \widehat{\ell_k}, \dots, \ell_i])\neq \mathrm{Max}(AB[\ell_1, \dots , \ell_i])$,
for all $1\leq k \leq i$. This shows that the assumption of Lemma \ref{Lemma2.3}(1) holds and so $\beta_{i,j}(R/I^2)\leq|\mathcal{L}_{i,j}|$.

Now suppose that $\tau=\{S^{\textbf{b}_{\ell_1}}, \dots , S^{\textbf{b}_{ \ell_{i+1}}}\}$ is a face of $\mathbb{L}^2(I)$ with
$$\mathrm{Sum}(\mathrm{Max}(AB[\ell_1, \dots , \ell_i]))=2di, \mathrm{Max}(AB[\ell_1, \dots, \ell_{i+1}])= \mathrm{Max}(AB[\ell_1, \dots , \ell_i]).$$ By means of the above explanation, we should have
$\tau=\{S_{\ell_1}S_{\ell'_1}, \dots , S_{\ell_{i+1}}S_{\ell'_{i+1}}\},$
when $\ell_k< \ell'_k$ for all $1\leq k \leq i+1$ and $(S_{\ell_k}\cup S_{\ell'_k})$s are disjoint for all $1\leq k \leq i$ and \linebreak $S_{\ell_{i+1}}\cup S_{\ell'_{i+1}}\subseteq \bigcup_{1\leq r \leq i}(S_{\ell_r}\cup S_{\ell'_r})$. Suppose that $1\leq k\leq i$. Then in view of the structure of $\mathbb{L}^2(I)$, $S_{\ell_k}S_{\ell'_k}\nmid S_{\ell_{i+1}}S_{\ell'_{i+1}}$. So there exists a variable $x_{r_k}$ such that

$\bullet$ either $x_{r_k}\in S_{\ell_k}\cup S_{\ell'_k}$  but $x_{r_k}\not\in S_{\ell_{i+1}}\cup S_{\ell'_{i+1}}$,

$\bullet$ or $x_{r_k}\in S_{\ell_k}\cap S_{\ell'_k}$  but $x_{r_k}\not\in S_{\ell_{i+1}}\cap S_{\ell'_{i+1}}$.

Hence

$\bullet$ either the $r_k$th row in $\mathrm{Max}(AB[\ell_1, \dots, \ell_i])$ is one or two, while the $r_k$th row in
$$\mathrm{Max}(AB[\ell_1, \dots, \widehat{\ell_k}, \dots, \ell_{i+1}]),$$
 is zero,

$\bullet$ or the $r_k$th row in $\mathrm{Max}(AB[\ell_1, \dots, \ell_i])$ is two, while the $r_k$th row in
$$\mathrm{Max}(AB[\ell_1, \dots, \widehat{\ell_k}, \dots, \ell_{i+1}]),$$ is one.

Therefore in each case we have
$$\mathrm{Max}(AB[\ell_1, \dots, \widehat{\ell_k}, \dots, \ell_{i+1}]) \neq \mathrm{Max}(AB[\ell_1, \dots, \ell_i]).$$
Thus by Lemma \ref{Lemma2.3}(2), $\beta_{i,j}(R/I^2)\geq|\mathcal{L}_{i,j}|$. These complete the proof.
\end{proof}

One can use Proposition \ref{3.6} for vanishing of the special graded Betti numbers of the second power of the edge ideal as follows.

\begin{corollary}\label{Cor3.7}
Suppose that $\mathcal{H}$ is a $d$-uniform hypergraph and $i>1$ such that \linebreak $\beta_{i,2di}(R/I^2)\neq 0$. Then $\mathcal{H}$ should have $2^i$ matchings of size $i$.
\end{corollary}
\begin{proof}
If $\beta_{i,2di}(R/I^2)\neq 0$, then in view of Proposition \ref{3.6}, $\mathcal{L}_{i,2di}\neq \emptyset$. So, there exists an element $e_\tau$ in $\mathcal{L}_{i,2di}$ of degree $2di$. As one can see in the proof of Proposition \ref{3.6}, $\tau$ should be in the form of $\{S_{\ell_1}S_{\ell'_1}, \dots , S_{\ell_i}S_{\ell'_i}\}$, when $\ell_k< \ell'_k$ and $(S_{\ell_k}\cup S_{\ell'_k})$s are disjoint for all $1\leq k \leq i$. Hence a set consisting precisely one edge from each set $\{S_{\ell_k}, S_{\ell'_k}\}$ for $1\leq k \leq i$, will be a matching of size $i$ in $\mathcal{H}$. Clearly, there exist $2^i$ such sets.
\end{proof}

Now we prove Theorem \ref{3.7}.
\begin{proof}[Proof of Theorem \ref{3.7}]
\begin{enumerate}
\item Suppose that $\mathcal{S}=\{S_1, \dots, S_i\}$ is a self semi-induced matching of type $(i,j)$ in $\mathcal{H}$. Then for each $1\leq \ell \leq i$, set $\tau_\ell=\{S_\ell^{t-1}S_j : 1\leq j \leq i\}$. (Note that if $t=1$, then $\tau_1 = \dots = \tau_i$.) By means of Lemma \ref{lemma3.3}, for each $1\leq \ell \leq i$, $\tau_\ell$ is a face of $\mathbb{L}^t(I)$.  We show that
$$\overline{e_{\tau_\ell}}\in(\mathrm{Ker}\overline{\partial}_{i}\setminus \mathrm{Im}\overline{\partial}_{i+1})_{|S_\ell|(t-1)+j}.$$
Firstly note that
\begin{align*}
\mathrm{deg}(e_{\tau_\ell})&=\mathrm{deg}(S_\ell^{t-1}\mathrm{lcm}(S_1, \dots , S_i))\\
&=\mathrm{deg}(S_\ell^{t-1})+\mathrm{deg}(\mathrm{lcm}(S_1, \dots, S_i))\\
&=|S_\ell|(t-1)+|\bigcup _{1\leq k \leq i} S_k |=|S_\ell|(t-1)+j.
\end{align*}
Without loss of generality we may assume that $\ell=1$ and for $1\leq r \leq i$, $\textbf{b}_r\in \mathcal{B}_{m,t}$ is the vector associated to the $r$th element of $\tau_1$. Now, in view of Remark \ref{Remarks2.2}(1), to prove $\overline{e_{\tau_1}}\in \mathrm{Ker}\overline{\partial}_{i}$ it is enough to show that for all $1\leq k \leq i$,
$$\mathrm{Max}(AB[1, \dots, i])\neq \mathrm{Max}(AB[1, \dots , \widehat{k}, \dots , i]).$$
Suppose that $2\leq k \leq i$ (resp. $k=1$). Since $\mathcal{S}$ is a self-matching, there is a vertex $x_{u_k}\in S_k\setminus \bigcup_{1\leq \ell \leq i, \ell \neq k}S_\ell$. Now, in view of Remark \ref{Remark3.2}(3), if $AB[1, \dots , i]=(c_{r,r'})_{n\times i}$, then $c_{u_k,k}=1$ (resp. $c_{\ell_k,k}=t$) and the other entries in the $u_k$th row are zero (resp. $t-1$). So the $u_k$th component in $\mathrm{Max}(AB[1, \dots , i])$ is one (resp. $t$), while the $u_k$th component in $\mathrm{Max}(AB[1, \dots , \widehat{k}, \dots , i])$ is zero (resp. $t-1$).

To prove the last part of our claim, suppose in contrary that $\overline{e_{1, \dots, i}}\in \mathrm{Im}\overline{\partial}_{i+1}$. Then by Remark \ref{Remarks2.2}(2), there should exist a face $\tau '$ containing $\tau_1$ such that
$$\mathrm{Max}(AB[1, \dots , i])=\mathrm{Max}(AB[1, \dots, i+1]).$$
 Note that since $\tau '$ is a face containing $S_1^t$, we should have
$$\tau'=\{S_1^t, S_1^{t-1}S_2, \dots , S_1^{t-1}S_i, S_1^{t-1}S_{i+1}\},$$
 by renumbering the edges if it is required. Hence $e_{\tau'}=e_{1, \dots, i+1}$, where $\textbf{b}_{i+1}$ is a vector whose the first component is $t-1$ and the $(i+1)$th component is one and other components are zero. Now, since $\mathcal{S}$ is a semi-induced matching, \linebreak $S_{i+1}\nsubseteq \bigcup_{1\leq \ell \leq i}
S_\ell$. Thus similar argument to above paragraph ensures the contradiction  $\mathrm{Max}(AB[1, \dots , i])\neq \mathrm{Max}(AB[1, \dots, i+1])$.

\item The proof of Part 1 implies 2.

\item immediately follows from Remark \ref{Remark3.2}(1) and Part 1.
\end{enumerate}
\end{proof}
In view of Remark \ref{Remark3.2}(1) and \cite[Example 5.2]{BBH}, the equalities in Theorem \ref{3.7}(3) holds for many classes of hypergraphs and there exists examples illustrating the strictness of the bound. 

The following corollary, which is an immediate consequence of \cite[Theorem 3.6]{Moradi+Khosh2017}, Theorem \ref{3.7}(3) and Remark \ref{Remark3.2}(1), can regain Theorem 2.4 in \cite{FS}. Recall that a \textbf{chain} in $\mathcal{H}$ is a sequence $v_0, S_1, v_1, \dots , S_k, v_k$ of vertices and edges in $\mathcal{H}$, where $v_i \in S_i$ for $1\leq i \leq k, v_i \in S_{i+1}$ for
$0 \leq i \leq k-1$. When $k>1$, we say that $\mathcal{H}$ is \textbf{$\mathcal{C}_k$-free} if it doesn't contain any chain $v_0, S_1, v_1, \dots , S_k, v_0$ with $k>1$ and distinct $v_i$s and $E_i$s.
\begin{corollary}\label{Cor3.8}
Let $\mathcal{H}$ be a $(C_2,C_5)$-free vertex decomposable hypergraph. Then
$$(s.s.i.m)'_\mathcal{H}\leq \mathrm{reg}(R/I) \leq (s.i.m)'_\mathcal{H}.$$
So for any $C_5$-free vertex decomposable
graph we regain $$\mathrm{reg}(R/I(G)) = (i.m)_G.$$
\end{corollary}

Note that the following example illustrates that the equality in Corollary \ref{Cor3.8} doesn't always hold.
\begin{example}\label{exam1}
Assume that  $\mathcal{H}$ is a $3$-uniform $(C_2,C_5)$-free vertex decomposable simple hypergraph with vertex set $\{x_1, \dots , x_9\}$ and edge set
 $$\{ \{x_1, x_2, x_3\}, \{x_4, x_5, x_6\}, \{x_7, x_8, x_9\}, \{x_1, x_4, x_7\}\}.$$
Then one can see that
$$(s.s.i.m)'_\mathcal{H}=4\leq \mathrm{reg}(R/I)=(s.i.m)'_\mathcal{H}=5.$$
\end{example}

The next example gives a class of hypergraphs achieving the equality in Theorem \ref{3.7}(3). This example demonstrates the advantage of our lower bound over those previously found.

\begin{example}\label{exam2}
Assume that $\mathcal{H}$ is a $d$-uniform simple hypergraph with 
$$V(\mathcal{H})=\{x_1, \dots , x_k\} \cup \bigcup_{1\leq i \leq s} \bigcup_{1\leq j \leq d-k}\{x_{i,j}\},$$
and edge set 
$$\mathcal{E}(\mathcal{H})=\{E_i=\{x_1, \dots , x_k, x_{i,1}, \dots, x_{i,d-k}\} : 1\leq i\leq s\},$$
 where $k$ is an integer with $1\leq k\leq d-1$.
One can check that $(s.s.i.m)'_\mathcal{H}=s(d-k-1)+k$, since $\mathcal{E}(\mathcal{H})$ forms the desired self semi induced matching in $\mathcal{H}$. Also if we set $u=x_1\dots x_k$ and $u_i=\prod_{1\leq j \leq d-k}x_{i,j}$ for every integer $i$ with $1\leq i \leq s$, then $I=u\langle u_1, \dots , u_s \rangle$. Thus since $\langle u_1, \dots, u_s\rangle $ is a complete intersection, in view of \cite[Lemma 4.4]{B+H+T}, for each $t\in \mathbb{N}$ we have
\begin{align*}
\mathrm{reg}(R/I^t)&=\mathrm{reg}(I^t)-1\\
&=\mathrm{reg}(u^t \langle u_1, \dots , u_s \rangle ^t)-1\\
&=\deg (u^t)+ \mathrm{reg}(\langle u_1, \dots , u_s \rangle ^t)-1\\
&= kt+(d-k)t+(d-k-1)(s-1)-1\\
&=d(t-1)+s(d-k-1)+k\\
&=d(t-1)+(s.s.i.m)'_\mathcal{H}.
\end{align*}
Also this is obvious that for large values of $s$, this bound is reasonably greater than the bound $d(t-1)+(d-1)(i.m)_\mathcal{H}$ given in \cite[Corollary 3.8]{BCKMS}.
\end{example}

Although there is no general upper bound known for regularity of powers of the edge ideals of arbitrary hypergraphs, but there is a good one for graphs in terms of matching number (\cite{B+B+H}). So using Macaulay 2 (\cite{Macaulay2}) for some examples of hypergraphs and some positive integers $t$ and also thanks to the general upper bound for graphs in \cite[Theorem 3.4]{B+B+H} yields to the following question.
\begin{question} \label{Question2}
Suppose that $t\in \mathbb{N}$. Then
\begin{itemize}
\item[1.] Is the following inequality holds for each $d$-uniform hypergraph $\mathcal{H}$ in which every two intersecting edges have exactly $d-1$ common vertices?
$$\mathrm{reg}(R/I^t)\leq d(t-1)+ m_{\mathcal{H}}(d-1).$$
\item[2.]  For which hypergraphs we have $\mathrm{reg}(R/I(\mathcal{H})^t)= d(t-1)+(s.s.i.m)'_{\mathcal{H}}$?
\end{itemize}
\end{question}

Note that the condition of Question \ref{Question2}(1) is necessary. For instance by example \ref{exam2} for $3$-uniform hypergraph $\mathcal{H}$ with $V(\mathcal{H})=\{x_1,\dots,x_5\}$ and $\mathcal{E}(\mathcal{H})=\{\{x_1, x_2, x_3\}, \{x_3, x_4, x_5\}\}$ we see that $\mathrm{reg}(R/I(H))=3$,
$\mathrm{reg}(R/I(\mathcal{H})^3)=9$ and $m_{\mathcal{H}}=1$ and so $d(t-1)+m_{\mathcal{H}}(d-1)$ can not be an upper bound even for the first power and $d$-uniform  well-known tree hypergraphs.

\textbf{Acknowledgement.}
The author is deeply grateful to Sara Faridi for drawing her attention to \cite{Faridi1} and \cite{Faridi2}. Also she would like to thank the referee for his/her valuable comments which substantially improved the quality of the paper.

\end{document}